\documentclass[12pt,letterpaper,reqno]{amsart}

\usepackage{times}

\usepackage[T1]{fontenc}

\usepackage{mathrsfs}

\usepackage{latexsym}

\usepackage[dvips]{graphics}

\usepackage{epsfig}

\usepackage{amsmath,amsfonts,amsthm,amssymb,amscd}

\input amssym.def

\input amssym.tex

\usepackage{color}

\usepackage{comment}

\usepackage{hyperref}

\usepackage{url}

\newcommand{\burl}[1]{\textcolor{blue}{\url{#1}}}

\newcommand\be{\begin{equation}}

\newcommand\ee{\end{equation}}

\newcommand\nbea{\begin{eqnarray*}}

\newcommand\neea{\end{eqnarray*}}

\newcommand\bea{\begin{eqnarray}}

\newcommand\eea{\end{eqnarray}}

\newcommand\bi{\begin{itemize}}

\newcommand\ei{\end{itemize}}

\newcommand\ben{\begin{enumerate}}

\newcommand\een{\end{enumerate}}

\newcommand\bc{\begin{center}}

\newcommand\ec{\end{center}}

\newcommand\ba{\begin{array}}

\newcommand\ea{\end{array}}

\newcommand{\E}{\ensuremath{\mathbb{E}}}

\newtheorem{thm}{Theorem}[section]

\newtheorem{cor}[thm]{Corollary}

\newtheorem{lem}[thm]{Lemma}

\newtheorem{defi}[thm]{Definition}

\newcommand{\twocase}[5]{#1 \begin{cases} #2 & \text{{\rm #3}}\\ #4
&\text{{\rm #5}} \end{cases}   }

\renewcommand{\v}{\text{Var}}

\newcommand{\hr}[1]{\href{#1}{\url{#1}}}

\numberwithin{equation}{section}

\title[Gaussian Distribution in Generalized Zeckendorf Decompositions in Small Intervals]{Gaussian Distribution of the Number of Summands in Generalized Zeckendorf Decompositions in Small Intervals}

\author[Best]{Andrew Best}
\email{\textcolor{blue}{\href{mailto:ajb5@williams.edu)}{ajb5@williams.edu}}}
\address{Department of Mathematics and Statistics, Williams College, Williamstown, MA 01267}

\author[Dynes]{Patrick Dynes}
\email{\textcolor{blue}{\href{mailto:pdynes@clemson.edu)}{pdynes@clemson.edu}}}
\address{Department of Mathematical Sciences, Clemson University, Clemson, SC 29634}

\author[edelsbrunner]{Xixi Edelsbrunner}
\email{\textcolor{blue}{\href{mailto:xe1@williams.edu}{xe1@williams.edu}}}
\address{Department of Mathematics and Statistics, Williams College, Williamstown, MA 01267}

\author[McDonald]{Brian McDonald}
\email{\textcolor{blue}{\href{mailto:bmcdon11@u.rochester.edu}{bmcdon11@u.rochester.edu}}}
\address{Department of Mathematics, University of Rochester, Rochester, NY 14627}

\author[Miller]{Steven J. Miller}
\email{\textcolor{blue}{\href{mailto:sjm1@williams.edu}{sjm1@williams.edu}},  \textcolor{blue}{\href{Steven.Miller.MC.96@aya.yale.edu}{Steven.Miller.MC.96@aya.yale.edu}}}
\address{Department of Mathematics and Statistics, Williams College, Williamstown, MA 01267}

\author[Tor]{Kimsy Tor}
\email{\textcolor{blue}{\href{mailto:ktor.student@manhattan.edu}{ktor.student@manhattan.edu}}}
\address{Department of Mathematics, Manhattan College, Riverdale, NY 10471}

\author[Turnage-Butterbaugh]{Caroline Turnage-Butterbaugh}
\email{\textcolor{blue}{\href{mailto:cturnagebutterbaugh@gmail.com}{cturnagebutterbaugh@gmail.com}}}
\address{Department of Mathematics, North Dakota State University, Fargo, ND 58102}

\author[Weinstein]{Madeleine Weinstein}
\email{\textcolor{blue}{\href{mailto:mweinstein@g.hmc.edu}{mweinstein@g.hmc.edu}}}
\address{Department of Mathematics, Harvey Mudd College, Claremont, CA 91711 }

\thanks{This research was conducted as part of the 2014 SMALL REU program at Williams College and was supported by NSF grants DMS1347804 and DMS1265673,  Williams College, and the Clare Boothe Luce Program of the Henry Luce Foundation. It is a pleasure to thank them for  their support, and the participants there and at the 16\textsuperscript{th} International Conference on Fibonacci Numbers and their Applications for helpful discussions.}



\date{\today}

\begin{document}

\begin{abstract}
Zeckendorf's theorem states that every positive integer can be written uniquely as a sum of non-consecutive Fibonacci numbers ${F_n}$, with initial terms $F_1 = 1, F_2 = 2$. Previous work proved that as $n \to \infty$ the distribution of the number of summands in the Zeckendorf decompositions of $m \in [F_n, F_{n+1})$, appropriately normalized, converges to the standard normal. The proofs crucially used the fact that all integers in $[F_n, F_{n+1})$ share the same potential summands and hold for more general positive linear recurrence sequences $\{G_n\}$.

We generalize these results to subintervals of $[G_n, G_{n+1})$ as $n \to \infty$ for certain sequences. The analysis is significantly more involved here as different integers have different sets of potential summands. Explicitly, fix an integer sequence $\alpha(n) \to \infty$. As $n \to \infty$, for almost all $m \in [G_n, G_{n+1})$ the distribution of the number of summands in the generalized Zeckendorf decompositions of integers in the subintervals $[m, m + G_{\alpha(n)})$, appropriately normalized, converges to the standard normal. The proof follows by showing that, with probability tending to $1$, $m$ has at least one appropriately located large gap between indices in its decomposition. We then use a correspondence between this interval and $[0, G_{\alpha(n)})$ to obtain the result, since the summands are known to have Gaussian behavior in the latter interval. \end{abstract}

\maketitle

\tableofcontents







\section{Introduction}

\subsection{Background}

Let $\{F_n\}$ denote the Fibonacci numbers, normalized so that $F_1 = 1$, $F_2 = 2$, and $F_{n+1} = F_n + F_{n-1}$. One of the more interesting, equivalent definitions of the Fibonacci numbers is that they are the unique sequence of positive integers such that every positive number has a unique legal decomposition as a sum of non-adjacent terms.\footnote{The requirement of uniqueness of decomposition forces us to start the sequence this way.} This equivalence is known as Zeckendorf's theorem \cite{Ze}. Once we know a decomposition exists, a natural question to ask is how the number of summands varies. The first result along these lines is due to Lekkerkerker \cite{Lek}, who proved that the average number of summands needed in the Zeckendorf decomposition of an integer $m \in [F_n, F_{n+1})$ is $\frac{n}{\varphi^2+1} + O(1)$, where $\varphi = \frac{1+\sqrt{5}}{2}$, the golden mean, is the largest root of the Fibonacci recurrence. More is true, and many authors have shown that the distribution of summands of $m \in [F_n, F_{n+1})$ converges to a Gaussian. These results have been extended to a variety of other sequences. There are several different methods of proof, from continued fractions to combinatorial perspectives to Markov processes; see \cite{B-AM,Day,DDKMV,DG,FGNPT,GT,GTNP,Ke,KKMW,LT,Len,MW1,MW2,Ste1,Ste2} for a sampling of results and methods along these lines, \cite{Al,CHMN1,CHMN2,CHMN3,DDKMMV,DDKMV} for generalizations to other types of representations, and \cite{BBGILMT,BILMT} for related questions on the distribution of gaps between summands.

The analysis in much of the previous work was carried out for $m \in [F_n, F_{n+1})$  (or, for more general sequences $\{G_n\}$, for $m \in [G_n, G_{n+1})$). The advantage of such a localization\footnote{As the sequence $\{F_n\}$ is exponentially growing,  it is easy to pass from $m$ in this interval to $m \in [0, F_n)$.} is that each $m$ has the same candidate set of summands and is of roughly the same size. The purpose of this work is to explore some of the above questions on a significantly smaller scale and determine when and how often we obtain Gaussian behavior. Note that we cannot expect such behavior to hold for all sub-intervals of $[F_n, F_{n+1})$, even if we require the size to grow with $n$. To see this, consider the interval \be [F_{2n} + F_{n} + F_{n-2} + \cdots + F_{\lfloor n^{1/4} \rfloor},\ F_{2n} + F_{n+1} + F_{\lfloor n^{1/4}\rfloor}). \ee The integers in the above interval that are less than $F_{2n} + F_{n+1}$ have on the order of $n/2$ summands, while those that are larger have at most on the order of $n^{1/4}$ summands. Thus the behavior cannot be Gaussian.\footnote{Though in this situation it would be interesting to investigate separately the behavior on both sides.}

In \cite{BDEMMTTW}, we proved Gaussian behavior for the number of summands in the Zeckendorf decomposition for almost all small subintervals of $[F_n, F_{n+1})$; in this work we generalize to other sequences. Henceforth $\{G_n\}$ will denote a positive linear recurrence sequence:

\begin{defi}\label{defn:goodrecurrencereldef}\label{def:goodrecurrence} A sequence $\{G_n\}_{n=1}^\infty$ of positive integers is a \textbf{Positive Linear Recurrence Sequence (PLRS)} if the following properties hold:

\ben
\item \emph{Recurrence relation:} There are non-negative integers $L, c_1, \dots, c_L$\label{c_i} such that \be G_{n+1} \ = \ c_1 G_n + \cdots + c_L G_{n+1-L},\ee with $L, c_1$ and $c_L$ positive.
\item \emph{Initial conditions:} $G_1 = 1$, and for $1 \le n < L$ we have
\be G_{n+1} \ =\
c_1 G_n + c_2 G_{n-1} + \cdots + c_n G_{1}+1.\ee
\een

A decomposition $\sum_{i=1}^{m} {a_i G_{m+1-i}}$\label{a_i} of a positive integer $N$ (and the sequence $\{a_i\}_{i=1}^{m}$) is \textbf{legal}\label{legal} if $a_1>0$, the other $a_i \ge 0$, and one of the following two conditions holds:

\begin{itemize}

\item Condition 1: 
We have $m<L$ and $a_i=c_i$ for $1\le i\le m$.

\item Condition 2: There exists $s\in\{0,\dots, L\}$ such that
\begin{equation}\label{eq:legalcondition2}
a_1\ = \ c_1,\ a_2\ = \ c_2,\ \cdots,\ a_{s-1}\ = \ c_{s-1}\ {\rm{and}}\ a_s<c_s,
\end{equation}
$a_{s+1}, \dots, a_{s+\ell} \ = \  0$ for some $\ell \ge 0$,
and $\{b_i\}_{i=1}^{m-s-\ell}$ (with $b_i = a_{s+\ell+i}$) is legal.

\end{itemize}

If $\sum_{i=1}^{m} {a_i G_{m+1-i}}$ is a legal decomposition of $N$, we define the \textbf{number of summands}\label{summands} (of this decomposition of $N$) to be $a_1 + \cdots + a_m$.
\end{defi}

Informally, a legal decomposition is one where we cannot use the recurrence relation to replace a linear combination of summands with another summand, and the coefficient of each summand is appropriately bounded. For example, if $G_{n+1} = 2 G_n + 3 G_{n-1} + G_{n-2}$, then $G_5 + 2 G_4 + 3 G_3 + G_1$ is legal, while $G_5 + 2 G_4 + 3 G_3 + G_2$ is not (we can replace $2 G_4 + 3 G_3 + G_2$ with $G_5$), nor is $7G_5 + 2G_2$ (as the coefficient of $G_5$ is too large). Note the Fibonacci numbers correspond to the special case of $L=2$ and $c_1 = c_2 = 1$.

Earlier work (see, for example, \cite{MW1,MW2}) proved that if $\{G_n\}$ is a PLRS then we again have unique legal decompositions and Gaussian behavior for the number of summands from $m \in [G_n, G_{n+1})$.

\subsection{Main Result}\hspace*{\fill} \\

Fix an increasing positive integer valued function $\alpha(n)$ with \be\label{eq:defnalphan} \lim_{n\to\infty}{\alpha(n)} \ = \ \lim_{n\to\infty}\left(n-\alpha(n)\right)\ = \ \infty.\ee  Our main result, given in the following theorem, extends the Gaussian behavior of the number of summands in Zeckendorf decompositions to smaller intervals. Note that requiring $m$ to be in $[G_n, G_{n+1})$ is not a significant restriction because given any $m$, there is always an $n$ such that this holds.

\begin{thm}[Gaussianity on small intervals] \label{thm:mainthm}
Let $\{G_n\}$ be a positive linear recurrence sequence with recurrence
\begin{align}
G_{n+1}\ = \ c_1G_n+c_2G_{n-1}+\cdots+c_LG_{n+1-L},
\end{align}
where we additionally assume $c_1\geq c_2\geq \cdots \geq c_L\geq 1$. Choose $\alpha(n)$ satisfying \eqref{eq:defnalphan}. The distribution of the number of summands in the decompositions of integers in the interval $[m,m+G_{\alpha(n)})$ converges to a Gaussian distribution when appropriately normalized for almost all $m\in [G_n,G_{n+1})$.
\end{thm}

In \S\ref{sec:preliminaries} we derive some useful properties of Zeckendorf decompositions, which we use in \S\ref{sec:proofofthmmain} to prove Theorem \ref{thm:mainthm}. The reason for our extra condition on the $c_i$'s surfaces in Lemma \ref{greedy}, where this constraint forces a truncated legal decomposition to remain legal. We conjecture that Theorem \ref{thm:mainthm} holds more generally.





\section{Preliminaries}\label{sec:preliminaries}

In order to prove Theorem \ref{thm:mainthm}, we establish a correspondence between the decompositions of integers in the interval $[m,m+G_{\alpha(n)})$ and those in $[0,G_{\alpha(n)})$.  Throughout this paper, when we write an interval $[a,b]$, we mean the integers in this interval.  We first introduce some notation.  Fix a non-decreasing positive function $q(n)$, taking on even values, such that \be\label{eq:defnqn} q(n) \ < \ n-\alpha(n), \ \ \ q(n) \ = \ o\left(\sqrt{\alpha(n)}\right), \ \ \ \lim_{n\to\infty} q(n) \ = \ \infty;  \ee the reason for the second condition is to allow us to appeal to known convergence results for a related system, as $q(n)$ will be significantly less than the standard deviation in that setting.

For $m\in [G_n,G_{n+1})$ with decomposition
\begin{align}
m & \ = \ \sum_{j=1}^n{a_jG_j},
\end{align}
define
\begin{align}
& C_1(m) \ := \ (a_1,a_2,...,a_{\alpha(n)}), \nonumber\\
& C_2(m) \ := \ (a_{\alpha(n)+1},...,a_{\alpha(n)+q(n)}),\nonumber\\
& C_3(m) \ := \ (a_{\alpha(n)+q(n)+1},...,a_n).
\end{align}
Note that each $a_i \in \{0,1\}$ for all $1 \le i \le n$. Let $s(m)$ be the number of summands in the decomposition of $m$. That is, let
\begin{align}
s(m) \ := \ \sum_{j=1}^n{a_j}.
\end{align}
Similarly, let $s_1(m), \ s_2(m)$, and $s_3(m)$ be the number of summands contributed by $C_1(m), \ C_2(m)$, and $C_3(m)$ respectively.

\begin{lem}
Let $x\in[m,m+G_{\alpha(n)})$. If there are at least $3L$ consecutive 0's in $C_2(m)$, then $C_3(x)$ is constant, and hence $s_3(x)$ is constant as well.
\end{lem}

\begin{proof}
Assume there are at least $3L$ consecutive 0's in $C_2(m)$. Then for some $k\in [\alpha(n)+3L,\alpha(n)+q(n))$, we have $a_{k-3L+1}=a_{k-3L+2}=\cdots=a_k=0$.  Let $m'$ denote the integer obtained by truncating the decomposition of $m$ at $a_{k-3L}G_{k-3L}$.  Then $m'<G_{k-3L+1}$. Since $G_{\alpha(n)}\leq G_{k-3L}$, it follows that for any $h<G_{\alpha(n)}$ we have
\begin{align}
m'+h  \ & < \ G_{k-3L+1}+G_{k-3L}\ \nonumber\\
& \leq \ c_1G_{k-3L+1}+c_2G_{k-3L}+\cdots+c_LG_{k-2L+2} \ = \ G_{k-3L+2},
\end{align}
and thus the decomposition of $m'+h$ has largest summand no greater than $G_{k-3L+2}$.  Then since $3L-2\geq L$, the Zeckendorf decomposition of $m+h$ is obtained simply by concatenating the decompositions for $m-m'$ and $m'+h$.  Hence $C_3(m+h)=C_3(m-m')=C_3(m)$.
\end{proof}

With this lemma, we see that the distribution of the number of summands involved in the decomposition of $x\in [m,m+G_{\alpha(n)})$ depends (up to a shift) only on what happens in $C_1(x)$ and $C_2(x)$, \textbf{\emph{provided}} that there is a gap between summands of length at least $3L$ somewhere in $C_2(m)$.  In light of this stipulation, we will show the following items in order to prove our main theorem.

\begin{itemize}

\item With high probability, $m$ is of the desired form (i.e., there is a gap between summands of length at least $3L$ in $C_2(m)$).

\item When $m$ is of the desired form, the distribution of the number of summands involved in $C_1(x)$ for $x\in [m,m+G_{\alpha(n)})$ converges to Gaussian when appropriately normalized.

\item The summands involved in $C_2(x)$ produce a negligible error term (i.e., there are significantly fewer summands from $C_2(x)$ than there are from $C_1(m)$).

\end{itemize}

We address the first point with the following lemma.

\begin{lem} As $n\to\infty$, with probability $1+o(1)$ there are at least $3L$ consecutive 0's in $C_2(m)$ if $m$ is chosen uniformly at random from the integers in $[G_n,G_{n+1})$.
\end{lem}


In \cite{BILMT}, Bower, Insoft, Li, Miller and Tosteson analyze the largest gap between summands in legal decompositions, assuming only that the $c_i$'s are non-negative (a less restrictive condition than we have). They prove that for $m \in [G_n, G_{n+1})$ as $n\to\infty$ with probability $1+o(1)$ there is a gap of length $C\log n$ for some constant $C>0$. This almost proves what we need, the only difficulty is we need the gap to be in $C_2(m)$. Instead of modifying their technical approach, in Appendix \ref{sec:appproofmodgaps} we derive a difference equation which proves as $n\to\infty$ with probability 1 we have a gap of any fixed finite size; while we prove that such a gap exists somewhere, the argument there is trivially modified to apply to just $C_2(m)$ (as if we have no summands in that region then the result is trivially true!), and completes the proof of the above lemma.

Until this point we have not used the condition that $c_1\geq c_2\geq\cdots\geq c_L$.  The following lemma illustrates the necessity of this condition.

\begin{lem}
If $c_1\geq c_2\geq \cdots\geq c_L$, then
\begin{align}
\sum_{j=1}^n{a_jG_{n+1-j}} \ \text{{\rm is\ legal}} \ \Rightarrow \ \sum_{j=2}^n{a_jG_{n+1-j}} \ \text{{\rm is\ legal.}}
\end{align}
\end{lem}

\begin{proof}
If the first sum is legal, then for some $s$ we have $a_1=c_1, \ a_2=c_2, \ \cdots \ ,a_s < c_s$.  Then we necessarily have $a_2\leq c_1, \ a_3\leq c_2, \ \cdots \ a_s'<c_{s'-1}$ for some $s'\leq s$.  We can repeat this process, noting that a block ends at $a_s$.  Therefore, the decomposition is legal.
\end{proof}

\begin{cor}\label{greedy}
If $c_1\geq c_2\geq \cdots \geq c_L$, then legal decompositions are chosen by the greedy algorithm.
\end{cor}

\begin{proof}
Suppose not, so that $\sum_{j=1}^n{a_jG_{n+1-j}}$ is a legal decomposition with $\sum_{j=k+1}^n{a_jG_{n+1-j}}\geq G_{n+1-k}$.  By an iterated application of the previous lemma, the above sum is a legal decomposition, and it is known that legal decompositions of integers in $[G_n,G_{n+1})$ all have largest term $G_n$, which gives a contradiction.
\end{proof}

Assuming $m$ is of the desired form (meaning there is a gap of length at least $3L$ in its decomposition), we now consider the distribution of $s(x)$ for $x\in [m,m+G_{\alpha(n)})$.

\begin{lem}\label{error}
If $m$ has at least 3L consecutive 0's in $C_2(m)$, then for all $x\in [m,m+G_{\alpha(n)})$, we have
\begin{align}
0\ \leq\ s(x)-s_3(m)-s(t(x)) \ < \ Kq(n),
\end{align}
where $K:=\max_j{c_j}$, and $t(x)$ denotes the bijection
\begin{align}
t:[m,m+G_{\alpha(n)})\to [0,G_{\alpha(n)})
\end{align}
given by
\begin{align}
\twocase{t(m+h)\ := \ }{m_0+h,}{if $m_0+h<G_{\alpha(n)},$}{m_0+h-
G_{\alpha(n)},}{if $m_0+h\geq G_{\alpha(n)}$,}
\end{align}
where $m_0$ is the sum of the terms in the decomposition of $m$ truncated at $a_{\alpha(n)-1}G_{\alpha(n)-1}$.
\end{lem}

\begin{proof} First, note that the number of summands in the decomposition of $x$ with indices $i\in [\alpha(n)$, $\alpha(n)$ $+$ $q(n))$ must be less than $Kq(n)$. Next, we verify that $t$ from above in fact is a bijection, for which it suffices to show that it is an injection.  Note that if $t(m+h_1)=t(m+h_2)$, then since $|h_1-h_2|<G_{\alpha(n)}$, we must have $m_0+h_1=m_0+h_2$, so we may conclude that $t$ is injective, hence bijective since the domain and codomain are finite.  For any $x\in [m,m+G_{\alpha(n)})$, the decompositions of $t(x)$ and $x$ agree for the terms with index less than $\alpha(n)$ by virtue of Corollary \ref{greedy}.  Furthermore, the decompositions of $x$ and $m$ agree for terms with index greater than $\alpha(n)+q(n)$.  Therefore, the number of summands in the decomposition of $x$ with indices $i\in [\alpha(n),\alpha(n)+q(n))$ is equal to $s(x)-s_3(m)-s(t(x))$. Combining this with our initial observation, the lemma now follows.
\end{proof}

As a result of this lemma, the distribution of $s(x)$ over the integers in $[m,m+G_{\alpha})$ is a shift of its distribution over $[0,G_{\alpha(n)})$, up to an error bounded by $q(n)$.  With this fact, we are now ready to prove the main theorem.





\section{Proof of Theorem \ref{thm:mainthm}}\label{sec:proofofthmmain}

We now prove our main result. We assume below that $\{G_n\}$ is a PLRS with $c_1 \ge c_2 \ge \cdots \ge c_L \ge 1$, that $\alpha(n)$ obeys \eqref{eq:defnalphan}, and that $q(n)$ obeys \eqref{eq:defnqn}.

\begin{proof}[Proof of Theorem \ref{thm:mainthm}]
For a fixed $m\in [G_n,G_{n+1})$ with at least $3L$ consecutive 0's somewhere in $C_2(m)$ (which is true for almost all $m$ as $n\to\infty$), we define random variables $X_n$ and $Y_n$ by
\begin{align}
& X_n \ := \ s(H_1), \ \ \
Y_n \ := \ s(H_2),
\end{align}
where $H_1$ is chosen uniformly at random from $[m,m+G_{\alpha(n)})$ and $H_2=H_1-m$ (and thus is chosen uniformly at random from $[0,G_{\alpha(n)})$). Let
\begin{align}
& X_n' \ := \ \frac{1}{\sigma_x(n)}(X_n-\mathbb{E}[X_n]),
\intertext{and}
& Y_n' \ := \ \frac{1}{\sigma_y(n)}(Y_n-\mathbb{E}[Y_n]),
\end{align}
where $\sigma_x(n)$ and $\sigma_y(n)$ are the standard deviations of $X_n$ and $Y_n$, respectively, so that $X_n$ and $Y_n$ are normalized with mean 0 and variance 1.  It is known that $Y_n'$ converges to a Gaussian distribution with mean and variance of order $\alpha(n)$ (see, for example, \cite{MW1}), and we claim that $X_n'$ converges to a Gaussian distribution as well.

Note that $X_n=Y_n+C+e(n)$, where $C$ is a constant and $e(n)$ is an error term with $|e(n)|<Kq(n)$.  From the assumption that $q(n)=o\left(\sqrt{\alpha(n)}\right)$ we have $q(n)/\sigma_y(n) \to 0$.  Therefore, we have
\begin{align}
\v[X_n] \ & = \ \v[Y_n]+\v[e(n)]+2\text{Cov}[Y_n,e(n)] \nonumber \\
& \leq \ \v[Y_n]+4K^2q(n)^2+2\E[(Y_n-\E[Y_n])(e(n)-\E[e(n)])] \nonumber \\
& \leq \ \v[Y_n]+4K^2q(n)^2+2\E[|Y_n-\E[Y_n]|\cdot |e(n)-\E[e(n)]|] \nonumber \\
& \leq \ \v[Y_n]+4K^2q(n)^2+4Kq(n)\E[|Y_n-\E[Y_n]|] \nonumber \\
& \leq \ \v[Y_n]+4K^2q(n)^2+4Kq(n)\E[(Y_n-\E[Y_n])^2]^{1/2} \nonumber \\
& = \ \v[Y_n]+4K^2q(n)^2+4Kq(n)+4Kq(n)\sigma_y(n) \nonumber \\
& = \ \v[Y_n](1+o(1)),
\end{align}
Hence $\sigma_x(n)=\sigma_y(n)(1+o(1))$.  Note that $|X_n-Y_n-C|<Kq(n)$, so $|X_n'\sigma_x(n)-Y_n'\sigma_y(n)|=O(q(n))$, and thus
\begin{align}
|X_n'-Y_n'|\  = \ \left|X_n'\frac{\sigma_x(n)}{\sigma_y(n)}-Y_n'\right|(1+o(1)) \ = \ O\left(\frac{q(n)}{\sigma_y(n)}\right)\ = \ o(1).
\end{align}

Let $A_n$ and $B_n$ be the cumulative distribution functions for $X_n'$ and $Y_n'$, respectively.  By Lemma \ref{error} combined with the above bound on $|X_n'-Y_n'|$, we have
\begin{align}
B_n(x-o(1)) \ \leq \ A_n(x) \ \leq \ B_n(1+o(1)).
\end{align}

Since $\{B_n\}_n$ converges pointwise to the cumulative distribution function for a Gaussian distribution, say $B(x)$, and it is easy to show that if $\beta_n\to 0$, then $B_n(x+\beta_n)\to B(x)$ since $B$ is continuous and $B_n$ is monotone non-decreasing, it follows that $\{A_n\}_n$ also converges pointwise to $B(x)$.  This completes the proof.
\end{proof}





\section{Conclusion and Future Work}

By finding a correspondence between generalized Zeckendorf decompositions in the interval $[m, m + G_{\alpha(n)})$ and in the interval $[0, G_{\alpha(n)})$, we are able to prove convergence to Gaussian behavior on many sub-intervals. The key step is to show that almost surely an integer $m$ chosen uniformly at random from $[G_n, G_{n+1})$ permits the construction of a bijection onto the interval $[0, G_{\alpha(n)})$. Our results then follow from previous work on the Gaussian behavior of the number of generalized Zeckendorf summands in this interval.

In the future, we plan to extend our results to more general sequences. The first natural candidate is to remove the assumption on the $c_i$'s among the positive linear recurrence sequences we study. Other interesting topics include the signed decompositions (or far difference representations) where both positive and negative summands are allowed (see \cite{Al, DDKMV}), the $f$-decompositions of \cite{DDKMMV}, and some other recurrences where the leading coefficient is zero (which in some cases leads to a loss of unique decompositions), such as \cite{CHMN1, CHMN2, CHMN3}.

\appendix






\section{Elementary Proof of Moderate Gaps}\label{sec:appproofmodgaps}

A crucial ingredient in our proof is that almost surely the Zeckendorf decomposition of an $m \in [G_n, G_{n+1})$ has a gap of length $Z$ or more for some fixed $Z > 3L$. This follows immediately from the work of Beckwith, Bower, Gaudet, Insoft, Li, Miller and Tosteson \cite{BBGILMT} (see also \cite{B-AM}), who showed that almost surely the longest gap is of the order $\log n$. It is possible to elementarily prove this result by deriving a recurrence relation for these probabilities and analyzing its growth rate directly, which we do below both in the hopes that it might be of use in related problems, and also to keep the argument elementary.

\begin{thm} Let $\{G_n\}$ be a PLRS with recurrence relation with positive integer coefficients: \be G_{n+1} \ = \ c_1 G_n + \cdots + c_L G_{n-L+1}, \ \ \ c_i, L \ge 1. \ee Fix an integer $Z > L$ and let $H_{n+1}$ be the number of integers $m\in [0, G_{n+1})$ such that $m$'s legal decomposition does not have a gap between summands of length $Z$ or greater.\footnote{Note that we only care about a gap of length at least $Z$ between adjacent summands; thus our decomposition may miss many elements before the first chosen summand. It is also easier to first count in $[0, G_{n+1})$ and later shift to $[G_n, G_{n+1})$ by subtraction.} Then $H_n/G_n$ converges to zero exponentially fast. \end{thm}

\begin{proof} We first find a recurrence relation for $\{H_n\}$. Consider $n$ much larger than $L+Z$, and let $m\in [0, G_{n+1})$ be arbitrary. We count how many $m$ have a legal decomposition with no gap between summands of length $Z$ or more by looking at the possible beginning strings of $m$'s decomposition. Specifically, we look at how often $G_n, G_{n-1}, \dots, G_{n-L+1}$ occur. As the analysis is trivial when $L=1$ we assume $L \ge 2$ below.

\begin{itemize}

\item \textbf{Case $n$: At most $c_1-1$ summands of $G_n$:} We may have 0, 1, $\dots$, $c_1-1$ or $c_1$ occurrences of $G_n$. If we have zero $G_n$'s, then the number of possible completions of the decomposition such that all gaps are less than $Z$ is, by definition, $H_n$. If we have 1, $\dots$, $c_1-1$ summands of $G_n$ then the number of completions is $H_n - H_{n-Z}$ (we must subtract $H_{n-Z}$ as we have a summand, and we cannot have $Z$ or more non-chosen summands now). Thus the contribution from this case to $H_{n+1}$ is \be H_n + (c_1-1)(H_n - H_{n-Z}) \ = \ c_1 (H_n - H_{n-Z}) + H_{n-Z}.\ee

\item \textbf{Case $n-1$: At most $c_2-1$ summands of $G_{n-1}$:} To be in this case, we first have $c_1$ summands of $G_n$. If we have no $G_{n-1}$ terms in our decomposition, then we have $H_{n-1} - H_{n-1-(Z-1)} = H_{n-1} - H_{n-Z}$ ways to complete the decomposition (as we have zero summands, we must be careful and avoid not taking any of the next $Z-1$ summands). If we have 1, $\dots$, $c_2-1$ copies of $G_{n-1}$ then there are $H_{n-1} - H_{n-1-Z}$ ways to complete the decomposition as desired. Thus this case contributes to $H_{n+1}$ \bea & &  (H_{n-1} - H_{n-Z}) + (c_2-1)(H_{n-1} - H_{n-1-Z}) \nonumber\\ & & \ \ \ \ \ \ = \ c_2 (H_{n-1} - H_{n-1-Z}) + (H_{n-1-Z}-H_{n-Z}). \eea

\item \textbf{Case $n-\ell$: At most $c_\ell-1$ summands of $G_{n-\ell+1}$:} Similar to the earlier cases, if we have zero copies of $G_{n-\ell+1}$ then there are $H_{n-\ell}-H_{n-\ell-(Z-1)} = H_{n-\ell}-H_{n-\ell+1-Z}$ ways to complete the decomposition, while if there are 1, $\dots$, $c_\ell-1$ copies of $G_{n-\ell+1}$  there are $H_{n-\ell}-H_{n-\ell-Z}$ possibilities. Thus the total contribution to $H_{n+1}$ is \bea & & (H_{n-\ell}-H_{n-\ell+1-Z}) + (c_\ell-1)(H_{n-\ell}-H_{n-\ell-Z})\nonumber\\ & & \ \ \ \ \ \ = \ c_\ell (H_{n-\ell}-H_{n-\ell-Z}) + (H_{n-\ell-Z}-H_{n-\ell+1-Z}).\eea If $\ell < L$ we continue to the next case with now $c_\ell$ copies of $G_{n-\ell+1}$, while if $\ell=L$ then the process terminates, as the recurrence relation and our definition of legality  means we would replace this beginning string with $c_1 G_n + \cdots + c_L G_{n-L+1}$ with $G_{n+1}$, contradicting both the fact that we have a legal decomposition and that our number is less than $G_{n+1}$.

\end{itemize}

Combining the above, we obtain the recurrence for $\{H_n\}$ (notice that we have a telescoping sum) \be\label{eq:recurrenceforHnplus1} H_{n+1} \ = \ c_1 (H_n - H_{n-Z}) + \cdots + c_L (H_{n-L+1} - H_{n-L+1-Z}) + H_{n-L+1-Z}. \ee Setting $\widetilde{H}{_{n+1}} = H_{n+1} - H_n$, we see $\widetilde{H}{_{n+1}}$ counts the number of integers in $[G_n, G_{n+1})$ with no gaps of length $Z$ or more, and using \eqref{eq:recurrenceforHnplus1} with $n$ and $n-1$ gives us \eqref{eq:recurrenceforHnplus1} with each $H_k$  replaced with $\widetilde{H}_k$.  As $\{\widetilde{H}_n\}$ is a strictly increasing sequence, the largest eigenvalue of its recurrence must be greater than 1. From the Generalized Binet Formula  expansion for $H_n$  (see, for example, Theorem A.1 of \cite{BBGILMT}), there is some constant $\alpha \in (0,1)$ such that $H_{n-\ell-Z} \ge \alpha H_{n-\ell}$. Letting \be \twocase{\widetilde{c}{_\ell} \ = \ }{c_\ell(1-\alpha)}{if $1 \le \ell < L$}{c_L}{if $\ell = L$,} \ee we find \bea \widetilde{H}_{n+1} \ < \ \widetilde{c}{_1} \widetilde{H}_1 + \cdots + \widetilde{c}{_L} \widetilde{H}_{n-L+1}. \eea

Notice this is almost the same recurrence as that of $G_n$, the only difference being that at least one of the coefficients $\widetilde{c}{_\ell}$ is smaller than $c_\ell$. To prove $\{\widetilde{H}_n\}$ grows exponentially slower than $\{G_n\}$, we instead study the sequence $\{\widehat{H}_n\}$ which satisfies the recurrence \bea \widehat{H}_{n+1} \ = \ \widetilde{c_1}\widehat{H}_n + \cdots + \widetilde{c_L}\widehat{H}_{n-L+1}, \eea as $\widetilde{H}_n \le \widehat{H}_n$.

We use many standard properties of the Generalized Binet Formula expansions below; see Theorem A.1 of \cite{BBGILMT} for statements and proofs. The solution to a linear recurrence of fixed, finite length is of the form \be \beta_{1,r_1} n^{r_1} \lambda_1^n + \cdots + \beta_{1,0} \lambda_1^n + \cdots + \beta_{k,r_k} n^{r_k} \lambda_k^n + \cdots + \beta_{k,0} \lambda_k^n. \ee As the $c_i$'s are positive, there is a unique positive  root $\lambda > 1$ for $\{G_n\}$, and all other roots are less than 1 in absolute value. Thus $G_n = \beta \lambda^n + O(|n^r \lambda_2|^n)$ for some $0 < |\lambda_2| < \lambda$ (if $\beta=0$ then $G_n$ is exponentially decaying). As all the coefficients of the recurrence for $\{\widehat{H}_n\}$ are positive, there is also a unique root of largest absolute value.

We claim $\omega < \lambda$. Clearly $\omega \le \lambda$ as this sequence grows slower; if they were equal then $\omega$ would be a root of both characteristic polynomials, and we would find \bea \omega^L \ = \  \widetilde{c}_1 \omega^{L-1} + \cdots + \widetilde{c}_L \ < \ c_1 \omega^{L-1} + \cdots + c_L \ = \ \omega^L, \eea a contradiction. Thus $\widehat{H}_n$ grows exponentially slower than $G_n$, which completes the proof.\end{proof}





\ \\

\end{document}